\documentclass[11pt, reqno]{amsart}

\usepackage{amsmath, amsfonts, amssymb, amsbsy, bbm, bigstrut, graphicx, enumerate,  upref, longtable, comment, comment,xcolor}
\usepackage{float}
\usepackage{tabularx, booktabs, makecell, caption}
\usepackage{siunitx}
\usepackage[breaklinks]{hyperref}

\usepackage[numbers, sort&compress]{natbib} 
\usepackage{comment}
\usepackage[export]{adjustbox}

\newcommand{\cq}{\mathcal{Q}}
\newcommand{\cp}{\mathcal{P}}

\newtheorem{thm}{Theorem}[section]
\newtheorem{lmm}[thm]{Lemma}

\newtheorem{defn}[thm]{Definition}

\theoremstyle{definition}
\newtheorem{remark}[thm]{Remark}
\newtheorem{ex}[thm]{Example}

\newcommand{\argmin}{\operatorname{argmin}}

{\tiny {\tiny }}

\newcommand{\ee}{\mathbb{E}}

\newcommand{\pp}{\mathbb{P}}

\newcommand{\rr}{\mathbb{R}}

\newcommand{\tr}{\operatorname{Tr}}

\newcommand{\var}{\mathrm{Var}}
\newcommand{\ve}{\varepsilon}

\newcommand{\N}{\mathbb{N}} 

\newcommand{\fpar}[2]{\frac{\partial #1}{\partial #2}}

\newcommand{\rank}{\operatorname{rank}}

\makeatletter
\def\thickhline{%
	\noalign{\ifnum0=`}\fi\hrule \@height \thickarrayrulewidth \futurelet
	\reserved@a\@xthickhline}
\def\@xthickhline{\ifx\reserved@a\thickhline
	\vskip\doublerulesep
	\vskip-\thickarrayrulewidth
	\fi
	\ifnum0=`{\fi}}
\makeatother

\newlength{\thickarrayrulewidth}
\numberwithin{equation}{section}

\renewcommand{\hat}{\widehat}
\renewcommand{\tilde}{\widetilde}

\begin{document}
\title[Matrix completion]{Matrix completion with data-dependent missingness probabilities}

\author[S. Bhattacharya]{Sohom Bhattacharya}
\address{S. \ Bhattacharya\hfill\break
	Department of Statistics\\ Stanford University\\ California, CA 94305, USA.}
\email{sohomb@stanford.edu}
\author[S. Chatterjee]{Sourav Chatterjee}
\address{S. \ Chatterjee\hfill\break
	Departments of Mathematics and Statistics\\ Stanford University\\ California, CA 94305, USA.}
\email{souravc@stanford.edu}

\begin{abstract}
The problem of completing a large matrix with lots of missing entries has received widespread attention in the last couple of decades. Two popular approaches to the matrix completion problem are based on singular value thresholding and nuclear norm minimization. Most of the past works on this subject assume that there is a single number $p$ such that each entry of the matrix is available independently with probability $p$ and missing otherwise. This assumption may  not be realistic for many applications. In this work, we replace it with the assumption that  the probability that an entry is available is an unknown function $f$ of the entry itself. For example, if the entry is the rating given to a movie by a viewer, then it seems plausible that high value entries have greater probability of being available than low value entries. We propose two new estimators, based on singular value thresholding and nuclear norm minimization, to recover the matrix under this assumption. The estimators involve no tuning parameters, and are shown to be consistent under a low rank assumption. We also provide a consistent estimator of the unknown function $f$.
\end{abstract}

\maketitle

\section{Introduction}
Let $M$ be an $m \times n$ matrix, which is only partially observed, possibly with added noise. Given an estimate $\hat{M}$ of  $M$, we define its mean squared error~as
\begin{equation}\label{eq:mse}
	\text{MSE}(\hat{M}) := \ee \biggl[\frac{1}{mn}\sum\limits_{i=1}^{m}\sum\limits_{j=1}^{n}\left(\hat{m}_{ij}-m_{ij}\right)^2\biggr],
\end{equation}
where $m_{ij}$ and $\hat{m}_{ij}$ denote the $(i,j)$-th entries of $M$ and $\hat{M}$ respectively. Given a sequence of such estimation problems, where $M_k$ and $\hat{M}_k$ denote the parameter and estimator matrices of the $k$-th problem, we call the sequence of estimators $\hat{M}_k$ consistent if 
\[
\lim\limits_{k \rightarrow \infty} \text{MSE}(\hat{M}_k)=0.
\]
Estimating a large matrix from a few randomly selected (and possibly noisy) entries is a common objective in many statistical problems. The basic assumption in all of the work in this area is that the matrix has either low rank or is approximately of low rank in some suitable sense. Some of the prominent applications of matrix completion include compressed sensing~\cite{cand1,don1, cand2,cand3,cand4,cand5}, collaborative filtering~\cite{bill,ren}, multi-class learning~\cite{mcl1,mcl2}, dimension reduction~\cite{dim1,dim2} and subspace estimation~\cite{yuxin1}. Theoretical guarantees of matrix completion under various assumptions have been worked out in \cite{am1, davplan, don2, eng1, eng2, kol1, kol2, klopp1, klopp2, klopp3, kmo, mjw, mona, rohde,rm}. This is only a small sampling of the huge literature on this topic. For a recent survey, see~\cite{nks}.  

In many of the above works, it is assumed that the entries are missing uniformly at random. This may not be a realistic assumption in many applications. For example, in the classic problem of movie ratings, if a particular movie gets poor reviews, fewer numbers of viewers are expected to review it and hence the probability of missing entries corresponding to that particular movie would be higher. Work on matrix completion under the `missing not at random' (MNAR) assumption is relatively sparse. Some examples include deterministic missing patterns or missing patterns that depend on the matrix, using spectral gap conditions \cite{bhj}, rigidity theory \cite{sin}, algebraic geometry \cite{kiraly} and other methods \cite{lee,pim,sportisse}. For random but non-uniform missing patterns, a variety of statistical guarantees for procedures based on nuclear-norm penalization and other ideas are available~\cite{klopp3,foygel,kol2, cand3,rohde,mjw, don2, davplan, zhupca,choasym, mamnar}.  However, these guarantees almost always require a careful choice of the penalty parameter (or some other parameter, such as rank) based on knowledge about the unknown matrix that is unlikely to be available. This is in contrast to the case of uniform missing pattern, where we now have many algorithms that assume no knowledge of the unknown matrix.


In the present work, we assume that the probability of an entry being revealed is a function $f$ of the value of that entry, and the revealed entries are allowed to be noisy. This frequently encountered example of missingness where a variable governs its own missingness is known as {\it self-masking MNAR}~\cite{mohan2018estimation}. Under these assumptions, we provide an estimator of the parameter matrix based on a spectral method and prove its consistency under a low rank assumption. We also provide a second estimator based on nuclear norm minimization. This estimator performs significantly better than the spectral estimator in the absence of noise, but may not work well for noisy entries. Moreover, it is computationally expensive for large matrices. Lastly, we give estimates of the function $f$ using both methods, along with theoretical guarantees about it. Some numerical examples are worked out. The main advantage of our estimators is that they do not involve penalty parameters (or any other user-specified parameters) which have to be carefully chosen to ensure that the theoretical guarantees work out.  The cost is that we have asymptotic consistency results rather than finite sample error bounds.

A recent paper that works under the setting of self-masking MNAR, but in the setting of tensor completion, is~\cite{yang2021tenips}. In \cite{yang2021tenips}, the probabilities of missingness are called `propensity scores'. The main difference between \cite{yang2021tenips} (and similar papers) and our work is that in \cite{yang2021tenips}, it is assumed that the tensor of propensity scores is low-rank, while we make no such assumption. Indeed, one of the main observations in our paper, which we prove using spectral techniques reminiscent of the proof of Szemer\'edi's lemma in combinatorics, is that the matrix of propensity scores is guaranteed to have an approximately low rank structure under a Lipschitz assumption on $f$.

A natural extension of our work is to study beyond self-masking MNAR, namely, to consider examples where the process that causes the missingness
of an entry depends on multiple entries of the parameter matrix and not
only its value itself. Such directions are left for future research.

\section{Results}
\subsection{The problem}\label{setupsec}
Let $M$ be an $m \times n$ matrix with all entries in the interval $[-1,1]$. Let $f:[-1,1] \rightarrow [0,1]$ be a function. Let $X$ be a noisy version of $M$, modeled as a matrix with independent entries in $[-1,1]$, such that $\ee(x_{ij})=m_{ij}$ for each $i$ and $j$. The $(i,j)$-th entry of $X$ is revealed with probability $f(m_{ij})$, and remains hidden with probability $1-f(m_{ij})$, and these events occur independently. Our goal is to estimate $M$ using the observed entries of $X$. 





\subsection{Modified USVT estimator}
Our first proposal is an estimator of $M$ based on singular value thresholding. This is a modification of the Universal Singular Value Thresholding (USVT) estimator of~\cite{usvt}. The estimator is defined as follows: 
\begin{enumerate}
	\item Let $Y$ be the matrix whose $(i,j)$-th entry is $x_{ij}$ if the $(i,j)$-th entry of $X$ is revealed, and $0$ otherwise.
	\item Let $\sum \sigma_iu_iv^T_i$ be the singular value decomposition of $Y$. 
	\item Choose a positive number $\eta \in(0,1)$ and let  
	\[
	A = \sum_{i\, :\, \sigma_i \geq (2+\eta)\max\{\sqrt{m}, \sqrt{n}\}}\sigma_iu_iv^T_i.
	\]
	(In \cite{usvt}, it is recommended that $\eta$ be chosen to be $0.02$. For results concerning the optimal choice of the threshold, see \cite{don2}.)
	\item Truncate the entries of $A$ to force them to belong to the interval $[-1,1]$. Call the resulting  matrix $\hat{Q}$. 
	\item Let $P$ be the matrix whose $(i,j)$-th entry is $1$ if $x_{ij}$ is revealed, and $0$ otherwise.
	\item Repeat the above steps for the matrix $P$ instead of $Y$, to get $\hat{R}$.
	\item Define a matrix $W$ as $w_{ij} := \hat{q}_{ij}/\hat{r}_{ij}$ if $\hat{r}_{ij} \neq 0$, and $0$ otherwise. 
	\item Truncate the entries of $W$ to force them to be in $[-1,1]$. The resulting matrix is our estimator $\hat{M}$.
\end{enumerate}
The idea behind this estimator has some similarity with the one proposed recently by Ma and Chen~\cite{mamnar}, which is also based on a two-step procedure, first estimating the matrix of missingness probabilities and then using these estimated probabilities to estimate the unknown matrix. The algorithm of Ma and Chen involves a number of user-specified parameters, whereas ours does not, which may be a desirable feature.

Note that if the entries of $M$ and $X$ are known to belong to an interval $[a, b]$ instead of $[-1, 1]$, then subtracting  $(a + b)/2$ from each entry of X and dividing by $(b-a)/2$ forces the entries to lie in $[-1, 1]$. Then applying the above procedure, and finally multiplying the end-result by $(b- a)/2$ and adding $(a + b)/2$, we can get the desired estimate of $M$. The case of unknown $a,b$ is beyond the scope of the paper. Lastly, if $n>m$, one can simply work with the transpose of $X$ to get an estimate for the transpose of $M$.

\subsection{Modified Cand\`es--Recht estimator} 
Our second proposal is an estimator of $M$ based on nuclear norm minimization. This estimator works only in the absence of noise, so we assume that $X=M$. Let $\hat{M}$ be the matrix  that minimizes nuclear norm among all matrices that are equal to $M$ at the revealed entries, and have all entries in $[-1,1]$. (Recall that the nuclear norm of a matrix $M$, usually denoted by $\|M\|_*$, is the sum of its singular values.) Hence, given a set of observed entries $\Omega$, our estimator is obtained by solving the optimization problem:
\[
\hat{M}:= \argmin_{Z\in S} \|Z\|_*,
\]
where
\[
S := \{Z: (Z-M)_{ij}\mathbbm{1}_{(i,j)\in \Omega}=0, \, \|Z\|_\infty \le 1\}.
\]
This is a small modification of the popular Cand\`es--Recht estimator \cite{cand3, cand4, cand5}, suggested recently in \cite{c19}. The original estimator does not have the additional constraint that the entries of $\hat{M}$ have to be in $[-1,1]$. This extra constraint is not problematic since this is a convex constraint. For example, it can be easily implemented in R by adding an $\ell^\infty$ constraint using {\tt CVXR} package~\cite{cvxr2020}. Moreover, from an intuitive point of view, it makes sense to add this constraint since we already know that the entries of the unknown matrix $M$ are in $[-1,1]$. This estimator is similar to the one proposed by Klopp~\cite{klopp3}, except that our method does not involve a penalty parameter.  


\subsection{Consistency results}\label{consec}
We now state consistency results for the two estimators defined above. Suppose that we have a sequence of matrices $\{M_k\}_{k\ge 1}$, where $M_k$ has order $m_k\times n_k$, and $m_k,n_k\to \infty$ as $k\to \infty$. Let $\{X_k\}_{k\ge 1}$ be a sequence of random matrices with independent entries in $[-1,1]$ such that $\ee(X_k)=M_k$ for each $k$. In other words, $X_k$ is a noisy version of $M_k$. Let $\mathcal{M}$ be the union of the sets of entries of all of these matrices. Let $f:\mathcal{M}\to [0,1]$ be a function such that the noisy version of an entry with true value $m$ is revealed with probability $f(m)$, independently of all else. Note that it is irrelevant how $f$ is defined outside $\mathcal{M}$, which is why we took the domain of $f$ to be this countable set. 

Recall that a sequence of estimators $\{\hat{M}_k\}_{k\ge 1}$ is consistent if $\text{MSE}(\hat{M}_k) \to 0$ as $k\to \infty$, where $\text{MSE}$ stands for the mean squared error defined in equation \eqref{eq:mse}. We will now prove the consistencies of the two estimators defined above. The crucial assumption will be that the sequence $\{M_k\}_{k\ge 1}$ has {\it uniformly bounded rank}. This is a version of the frequently occurring {\it low rank assumption} from the literature. In addition to that, we will need some other technical assumptions. Our first result is the following theorem, which gives a sufficient condition for the consistency of the modified USVT estimator.
\begin{thm}\label{usvtthm}
In the above setup, suppose that the sequence $\{M_k\}_{k\ge 1}$ has uniformly bounded rank. Let $\mu_k$ be the empirical distribution of the entries of $M_k$. Suppose that for any subsequential weak limit $\mu$ of the sequence $\{\mu_k\}_{k\ge 1}$, there is an extension of $f$ to a Lipschitz function from $[-1,1]$ into $[0,1]$, also denoted by $f$, which has no zeros in the support of $\mu$. Then the modified USVT estimator based on $\{X_k\}_{k\ge 1}$ is consistent. 
\end{thm}

\begin{remark}
	The statement of the above Theorem is about asymptotic behavior of MSE. However, in our proofs, we obtain some finite sample error bounds which we have omitted, with the goal of increasing the readability of the result, and also for reducing the stringency of assumptions on $f$. In fact, the proof shows that if $\|M\|_{*}\le q \sqrt{mn}$ for some $q>0$, and $f\ge \delta$ everywhere for some $\delta>0$, and $f$ is a Lipschitz function with Lipschitz constant $L>0$, then for any $\varepsilon>0$, the MSE can be upper bounded by $$\frac{12}{\delta^2} \left(c_1 \min\biggl\{2\sqrt{\frac{r}{m}}+\varepsilon L+\sqrt{2\varepsilon (L+1)},\, 2\biggr\}+2c_2e^{-c_3n}\right),$$ where $r$ is a constant depending on $q$ and $\varepsilon$, and $c_1$, $c_2$, and $c_3$ are universal constants. Such a bound reveals how the magnitude of error is dependent on the nuclear norm of parameter matrix and the Lipschitz constant of $f$. 
\end{remark}
\begin{remark}
Note that in many examples, such as in most recommender systems, the matrix entries can only take values in a fixed finite set. In such examples, there is no loss of generality in the assumption that $f$ has an extension that is Lipschitz and nonzero everywhere on $[-1,1]$. Also, if $f$ is continuous and nonzero everywhere in $[-1,1]$, then the condition involving the empirical distribution of the entries is redundant.
\end{remark} 

The next theorem gives the consistency of the modified Cand\`es--Recht estimator, under the additional assumption that there is no noise. 
\begin{thm}\label{candesthm}
	In the above setup, suppose that the sequence $\{M_k\}_{k\ge 1}$ has uniformly bounded rank, and also suppose that $X_k=M_k$ for each $k$.   Let $\mu_k$ be the empirical distribution of the entries of $M_k$. Suppose that for any subsequential weak limit $\mu$ of the sequence $\{\mu_k\}_{k\ge 1}$, there is an extension of $f$ to a measurable function from $[-1,1]$ into $[0,1]$, also denoted by $f$, such that $f$ is nonzero and continuous almost everywhere with respect to $\mu$. Then the modified Cand\`es--Recht estimator is consistent for this problem.
\end{thm}
\begin{remark}
We will see in numerical examples that the modified Cand\`es--Recht estimator has superior performance. The advantage of the modified USVT estimator is twofold. First, it can be used when the matrix is very large, where using nuclear norm minimization may become infeasible due to computational cost. Second, in the presence of noise --- which is often the case in practice --- the modified Cand\`es--Recht estimator may perform badly, as we will see in the simulated and real data examples.
\end{remark}
\begin{remark}
Often, in many MNAR examples, identifiability of parameters is an issue (see, e.g.,~\cite{miao2016identifiability}), which corresponds to the notions that there might be two sets of parameter values which yield same observations and hence, the true parameter value cannot be identified. In Theorems \ref{usvtthm} and \ref{candesthm}, however, the fact that we are able to approximately recover the true matrix automatically implies that identifiability is not an issue, provided that the low rank assumption holds. (That is, if there are two candidates $M_1$ and $M_2$ for the true matrix, and they both have low rank, then our estimate $\hat{M}$ will be close to both $M_1$ and $M_2$ with high probability, which means that $M_1$ must be close to $M_2$.) 
\end{remark}


\subsection{Proof sketch}\label{pfsketch}
To prove Theorem \ref{usvtthm}, we first assume that $\mu_k$ converges weakly to a limit $\mu$ as $k\to \infty$.  Let $R_k$ be the matrix obtained by applying $f$ entrywise to $M_k$ and $Q_k$ be entrywise product of $M_k$ and $R_k$. Let $Y_k$ be the matrix obtained by replacing the unrevealed entries of $X_k$ by zero. Let $P_k$ be the matrix whose $(i,j)$-th entry is $1$ if the $(i,j)$-th entry of $X_k$ is revealed, and $0$ otherwise.

The main step is to show that $R_k$ and $Q_k$ are also approximately low rank matrices, in the sense that $\|R_k\|_*= o(m_k\sqrt{n_k})$ and  $\|Q_k\|_*= o(m_k\sqrt{n_k})$. This is proved using a spectral method, similar to the spectral proof of Szemer\'edi's regularity lemma. The key idea is that a low rank matrix is approximately a block matrix after a suitable permutation of rows and columns, and therefore, applying a Lipschitz function entrywise keeps it close to a block matrix, which, in  turn, is approximately low rank. Once this is established, it then follows by the standard results for USVT that  if $\hat{Q}_k$ and $\hat{R}_k$ are the estimates of $Q_k$ and $R_k$ obtained by applying the USVT algorithm to $Y_k$ and $P_k$, then $\hat{Q}_k \approx Q_k$ and $\hat{R}_k \approx R_k$ with high probability (in some appropriate sense). 



To prove Theorem \ref{candesthm}, we first show that one can possibly permute rows and columns in each $M_k$ to get an $L^2$ limit $W$. Next we prove there is a measurable function $V:[0,1]^2 \to [0,1]$ that is nonzero almost everywhere and $P_k$ converges to $V$ in cut distance almost surely subsequentially. This implies consistency of $\hat{M}_k$ by \cite[Theorem 2 and Theorem 3]{c19}.

\subsection{Estimating $f$}\label{fsec}
We will now produce an estimator for the unknown function $f$ that can be used with any consistent estimator. Our procedure is motivated by the nonparametric density estimation methods available in statistics literature. It is interesting to note, if the underlying function $f$ were indeed a constant function, we have observed from simulated examples that our estimator $\hat{f}^b$ is also close to a constant function. Hence, $\hat{f}^b$ can be used to check if the data are MNAR or not. The estimator involves the choice of a tuning parameter $b$, which is a positive integer, chosen by the user. Given a matrix $M$ with partially revealed entries as in Subsection \ref{setupsec}, and an estimator $\hat{M}$ of $M$, the estimator $\hat{f}^b$ of $f$ is defined as follows.
\begin{enumerate}
\item For $i=1,\ldots,2b+3$, let $c_i := -1 + (i-2)b^{-1}$. Note that this is a sequence of equally spaced points, starting at $c_1 = -1 - b^{-1}$ and going up to $c_{2b+3} = 1+b^{-1}$. 
\item For each $i$, choose $a_i$ uniformly at random from the interval $[c_i - (4b)^{-1}, c_i + (4b)^{-1}]$. 
\item In the interval $[a_i, a_{i+1}]$, define $\hat{f}^b$ to be the proportion of revealed entries among those entries of $M$ such that the corresponding entry of $\hat{M}$ is in $[a_i,a_{i+1}]$. 
\end{enumerate}
Note that the above procedure defines $\hat{f}^b$ on an interval that is slightly larger than $[-1,1]$, but that should not bother us, because the domain can then be restricted to $[-1,1]$. The following theorem gives a measure of the performance of $\hat{f}^b$ as an estimate of $f$.
\begin{thm}\label{fthm}
Suppose that $f$ is Lipschitz, with Lipschitz constant~$L$. Let $\mu$ be the empirical distribution of the entries of $M$ and $\theta:= \textup{MSE}(\hat{M})$. Then 
\begin{align*}
\int (\hat{f}^b(x)-f(x))^2 d\mu(x) &\le C \theta^{1/3}b^{5/3} + \frac{Cb}{mn} + \frac{CL^2}{b^2},
\end{align*}
where $C$ is a universal constant. 
\end{thm}
The above result shows that if $b$ is big, but much smaller than both $mn$ and $\theta^{-1/5}$, then $\hat{f}^b$ is close to $f$ at almost all entries of $M$. In practice, a good rule of thumb would be to choose $b$ such that $b$ is large, but at the same time, the intervals $[a_l,a_{l+1})$ contain substantial numbers of entries of~$\hat{M}$. One can try to choose $b$ optimally using some kind of cross-validation (such as leave-one-out cross-validation), but it may be hard to prove theoretical guarantees for such methods. 

Although our method of estimating $f$ has similarities with density estimation methods, the problem is quite different since the entries of the estimated matrix are not independent random variables --- in fact, they may have a complicated, or even intractable, dependence structure. One might wonder if traditional nonparametric methods of estimating $f$ can still be applied here under some smoothness constraint. Such questions are left for future investigation.

\subsection{Examples}
In this subsection we will see how the two estimators perform in some simulated examples and two real data examples. For real data examples, one should always check whether the matrix is low-rank approximable before applying our methods. Our simulations show taking $b$ of order $\sqrt{n}$ for estimating an $n \times n$ matrix yields good $\hat{f}$, although we do not have a theorem to prove that. Finally, one should also check if data is noisy or not, and should apply spectral estimator when noise is present.
\begin{ex}\label{ex:sim1}
	Consider a low rank $n\times n$ matrix $M$ with the entries of $M$ having marginal distribution $Uniform[-1,1]$. Here, we take $n=100$ and $\rank(M)= 7$. To generate such a matrix, we define $M_1=\sum_{i=1}^{6}d_iu_iv^T_i$, where: 
\begin{itemize}
\item For $i=1,\ldots,5$,  $d_i=2^{-i}$, and the components of $u_i$ and $v_i$ are i.i.d.~$Bernoulli(1/\sqrt{2})$ random variables.
\item $d_6=1$, $u_6$ is a vector of all $1$s, and $v_6$ has i.i.d $Uniform[0,2^{-5}]$ entries. 
\end{itemize}
It is not difficult to see that the entries of $M_1$ are i.i.d.~$Uniform[0,1]$ random variables. Multiplying each entry by $2$ and subtracting $1$, we get $M$. Then $M$ has rank $7$ with probability $1$, and the entries of $M$ are uniformly distributed in $[-1,1]$. We take $f(x)=0.5x^2 + .3$ to generate missing entries, and do not add noise. To obtain the modified Cand\`es--Recht estimator, we used code from the R package {\tt filling}~\cite{nucf}  and imposed the $\ell^\infty$ constraint using the {\tt CVXR} package~\cite{cvxr2020}. The modified USVT algorithm, being quite straightforward, was coded without the aid of existing packages.

The modified Cand\`es--Recht estimator was able to exactly recover the true $M$  almost all the time, resulting a very small MSE of order $10^{-9}$. The modified USVT estimator performed much worse, with an unimpressive MSE of $0.123$. The run-time of the modified USVT estimator was much lower than that of the modified Cand\`es--Recht estimator: $0.31$ seconds versus $4.08$ minutes. We will see in the next example that the performance of the modified USVT estimator becomes better when $n$ is larger, accompanied by a huge gain in run-time over the other estimator. We report both our estimators and their MSEs and run-times in Table \ref{tab:table1}. 
	\begin{table}[h]
		\caption {\label{tab:table1} Comparison table for Example \ref{ex:sim1}} 
	\begin{footnotesize}
		\begin{tabular}{ p{2cm} p{3cm} p{5cm}   }
			\toprule
			& {\small Modified USVT} & {\small Modified Cand\`es--Recht}\\
			\midrule
			MSE   & 0.123   &   $\sim 10^{-9}$\\
			Run-time &   0.31 sec     & 4.08 min\\
			\bottomrule	
		\end{tabular}
		\end{footnotesize}
	\end{table}

Next, for both estimators of $M$, we estimated $f$ using the method proposed in Section \ref{fsec}, taking $b=25$. The estimated $\hat{f}$'s are shown in Figure~\ref{fig:figure0}. As expected, the $\hat{f}$ based on the modified Cand\`es--Recht estimator has better performance.

	\begin{figure}[tbp]
	\begin{center}
			\includegraphics[scale=0.4]{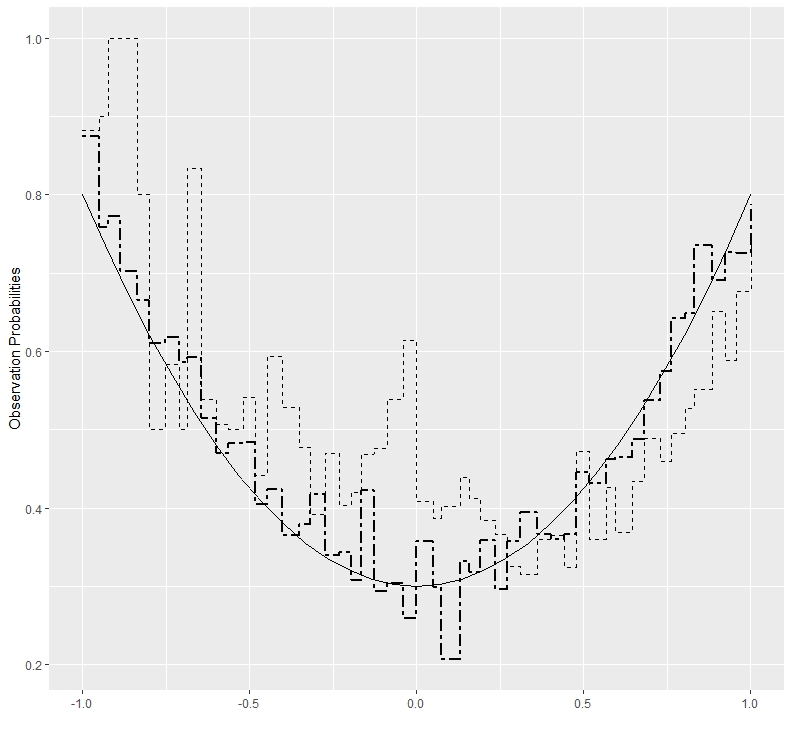}
			\caption{Estimates of $f$ in Example \ref{ex:sim1}. The dashed curve corresponds to modified USVT estimator, the double-dashed curve corresponds to the modified Cand\`es-Recht estimator, and the solid curve is the true $f$.}
			\label{fig:figure0}
	\end{center}
	\end{figure}
	
\end{ex}

\begin{ex}\label{ex:sim1.5}
	 Here, we want to see how our estimator performs as we vary the rank of underlying parameter matrix. To this end, we take parameter matrix same as previous example, with $n=500$ and choose rank $r=4,6,8,10$. We only report result of the modified USVT estimator, the results corresponding to modified Cand\'es- Recht estimator varies similarly. The MSE of the estimator as we vary rank are $0.007,0.012,0.012,0.033$ respectively.  The estimated $\hat{f}$ is shown in Figure~\ref{fig:figure1.5}. We also observe our estimator performs better than the vanilla USVT algorithm developed for MCAR. A comparison of MSE of the two estimators has been given below in Table \ref{tab:table1.5}.

	\begin{table}[h]
		\caption {\label{tab:table1.5} Comparison table of MSE for Example \ref{ex:sim1.5}} 
	\begin{footnotesize}
		\begin{tabular}{ p{2cm} p{3cm} p{5cm}   }
			\toprule
			Rank& {\small Modified USVT} & {\small Regular USVT}\\
			\midrule
			4   & 0.007   &   0.060\\
				6   & 0.012  &   0.063\\
					8   & 0.012  &   0.063\\
						10   & 0.033  &  0.062\\			
			\bottomrule
		\end{tabular}
	\end{footnotesize}
\end{table}

		\begin{figure}[tbp]
		\begin{center}
			\includegraphics[scale=0.5]{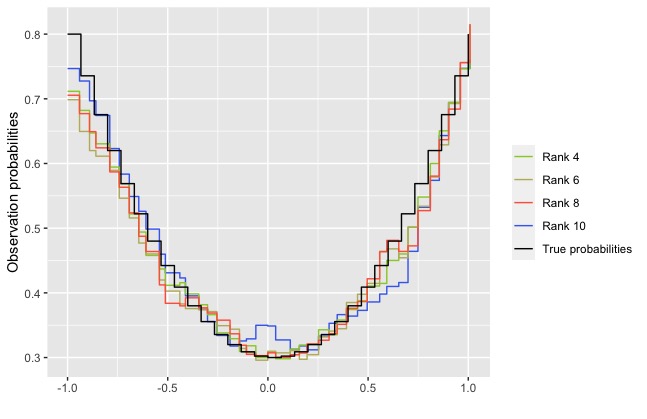}
			\caption{Estimates of $f$ in Example \ref{ex:sim1.5}.}
			\label{fig:figure1.5}
		\end{center}
	\end{figure}
\end{ex}

\begin{ex}\label{ex:sim2}
	This is the same as Example~\ref{ex:sim1}, but with $n=500$  to show how modified USVT has significant computation time advantage over the modified Cand\`es-Recht estimator. The MSE of the modified USVT estimator is now $0.011$, and that of the modified Cand\`es-Recht estimator is of order $10^{-9}$. So, with this larger sample size, the modified USVT estimator has reasonably good performance. The time to compute the modified USVT estimator $0.85$ seconds, whereas for the modified Cand\`es--Recht estimator, it is $2.51$ hours. This shows that even though the latter has much better performance in terms of MSE, it may be more practical to use the former if the matrix is large. We provide the estimators of $f$ in Figure \ref{fig:figure1}, taking $b=25$. We report the MSEs and run-times for both estimators in Table \ref{tab:table2}.
	\begin{table}[h]
		\caption {\label{tab:table2} Comparison table for Example \ref{ex:sim2}} 
	\begin{footnotesize}
		\begin{tabular}{ p{2cm} p{3cm} p{5cm}   }
			\toprule
			& {\small Modified USVT} & {\small Modified Cand\`es--Recht}\\
			\midrule
			MSE   & 0.011   &   $\sim 10^{-9}$\\
			Run-time &   0.85 sec     & 2.51 hrs\\
			\bottomrule
		\end{tabular}
	\end{footnotesize}
\end{table}
	
	A visual examination shows that both estimators perform well.
	
	\begin{figure}[tbp]
		\begin{center}
			\includegraphics[scale=0.4]{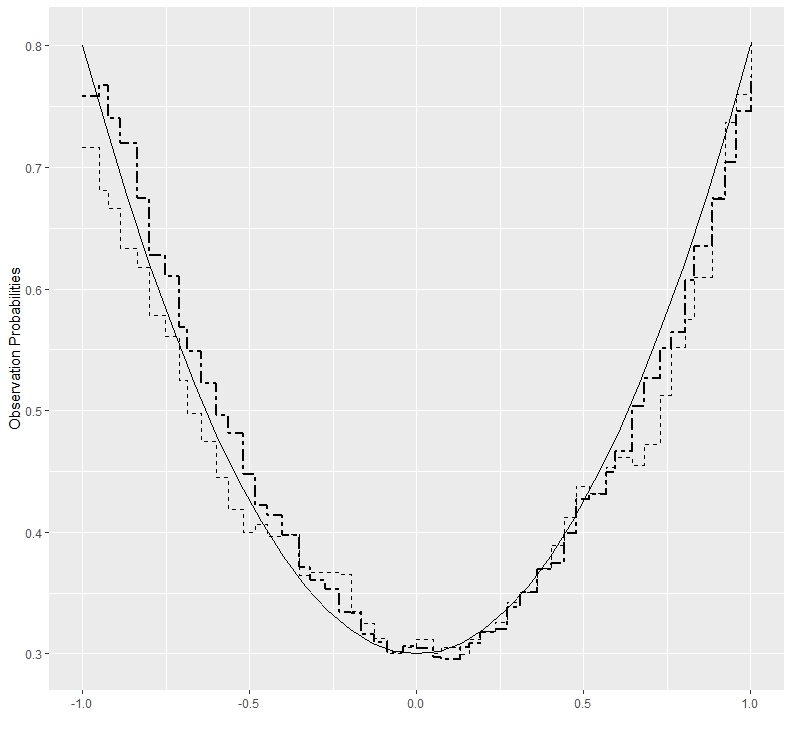}
			\caption{Estimates of $f$ in Example \ref{ex:sim2}. The dashed curve corresponds to modified USVT estimator, the double-dashed curve corresponds to the modified Cand\`es--Recht estimator, and the solid curve is the true $f$.}
			\label{fig:figure1}
		\end{center}
	\end{figure}
\end{ex}

\begin{ex}
	Under the same setup as before, we now show how the change of the parameter $b$, number of bins, affect the estimate of underlying function $f$. We choose $b=20,30,40,50$ and plot the resulting $\hat{f}$ in Figure~\ref{fig:figurebvary}. There does not seem to have much difference in $\hat{f}$ across different values of~$b$ .
	\begin{figure}[tbp]
		\begin{center}
			\includegraphics[scale=0.6]{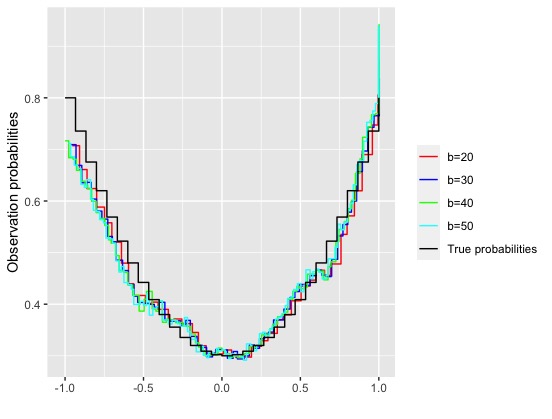}
			\caption{Estimating $\hat{f}$ under different values of $b$.}
			\label{fig:figurebvary}
		\end{center}
	\end{figure}
\end{ex}

\begin{ex}\label{ex:sim3}
	We will now show that the modified Cand\`es--Recht estimator performs poorly under presence of noise. Here, we take $n=100$ and $\rank(M)=2$, with the marginal distribution of the entries of $M$ being $Uniform[0,1]$, generated by the same procedure that we used to generate $M_1$ in Example~\ref{ex:sim1}. The noisy version of $M$, namely $X$, is generated as follows. For each $(i,j)$, generate $x_{ij} = 1$ with probability $m_{ij}$ and $x_{ij}=0$ with probability $1-m_{ij}$. Note that $\ee(x_{ij})=m_{ij}$. The entry $x_{ij}$ is revealed with probability $m_{ij}$, and remains hidden with probability $1-m_{ij}$ (that is, we took $f(x)=x$). For $n=100$, the MSE of modified USVT estimator turned out to be $0.017$, much better than the MSE of the modified Cand\`es--Recht estimator, which was $0.112$. The estimates of $f$ based on the two methods, with $b=10$, are depicted in Figure~\ref{fig:figure2}. The estimate based on the modified USVT method is reasonably good, even with $n$ as small as $100$  in this example. The estimate based on the modified Cand\`es--Recht estimator, however, is completely off: It estimates $f$ to be large near $0$ and $1$ and zero everywhere in between. This is because the observed entries consist solely of zeros and ones, and $\hat{M}$ coincides with $X$ at the observed values. So the estimation procedure for $\hat{f}$ deduces, incorrectly, that there is no chance of observing an entry if its non-noisy value is strictly between $0$ and $1$. 
	
		\begin{figure}[tbp]
		\begin{center}
			\includegraphics[scale=0.4]{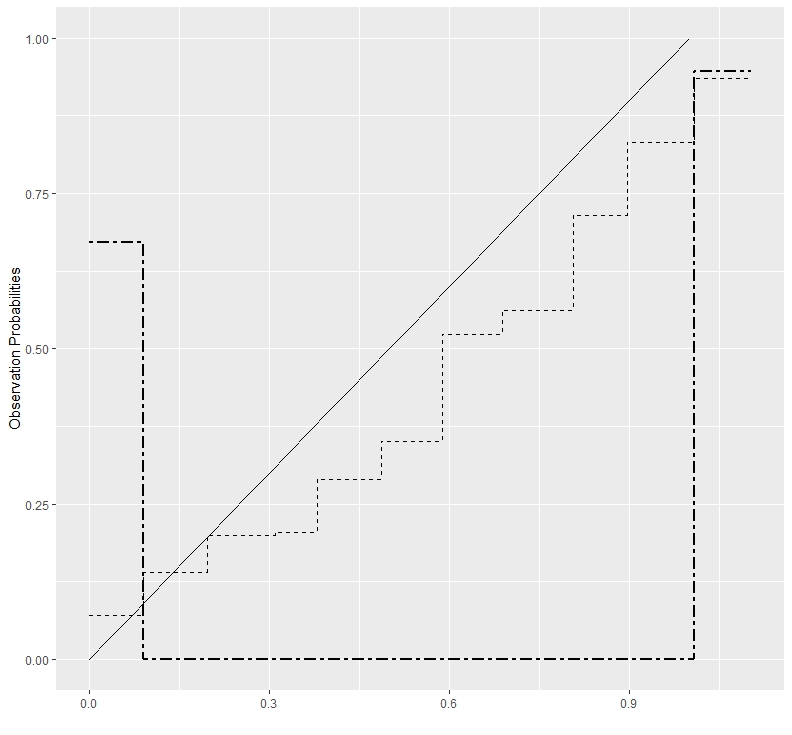}
			\caption{Estimation of $f$ in Example \ref{ex:sim3}. The dashed curve corresponds to modified USVT estimator, the double-dashed curve corresponds to the modified Cand\`es-Recht estimator, and the solid curve is the true $f$.}
			\label{fig:figure2}
		\end{center}
	\end{figure}
	
\end{ex}



\begin{ex}\label{ex:real1}
We now consider a real data example. In real data, it is not possible to  compare the performance of $\hat{f}$ with the  `true $f$', because we do not know what the true $f$ is (or if our model is actually valid). Still, if $\hat{f}$ turns out to be substantially different than a constant function, it validates the viewpoint that entries are not missing uniformly at random. We consider the well-known Jester data~\cite{jester1}, which consists of $100$ jokes rated by 73,421 users. The ratings are continuous values between $-10$ and $10$, entered by the users by clicking on an on-screen `funniness' bar. Not every user rates every joke, so there are many missing entries. Due to the prohibitively large run-time of the modified Cand\`es--Recht estimator, we first took a submatrix consisting of all $100$ jokes but a random sample of $300$ users.  Approximately $45 \%$ of the values were missing in this submatrix. The estimates of $f$ based on the two methods (with $b = 10$) are shown in Figure~\ref{fig:figurejes}. 

Interestingly, the two estimates are very different. We posit that this is due to the presence of noise in the observed matrix, which messes up the modified Cand\`es--Recht estimator. Indeed, the continuous nature of the ratings makes it very unlikely that the observed matrix is without noise. This is further validated by Figure~\ref{fig:figurejes2}, where we plot the percentage of the modified Cand\`es-Recht estimator matrix $\hat{M}$ that is captured by its rank-$k$ approximation, $k = 1,2,\ldots, 100$. (The percentage is simply the sum of squares of the top $k$ singular values divided by the sum of squares of all singular values.) This figure shows that to even get within $80\%$ of $\hat{M}$, we need to consider a rank-$25$ approximation. Thus, $\hat{M}$ is not of low rank, even approximately. This  invalidates the low rank assumption of the Cand\`es--Recht procedure, and allows us to conjecture that the $\hat{f}$ given by the modified USVT estimator is a better reflection of the true $f$, assuming that the model is correct.

\begin{figure}[tbp]
	\begin{center}
		\includegraphics[scale=0.4]{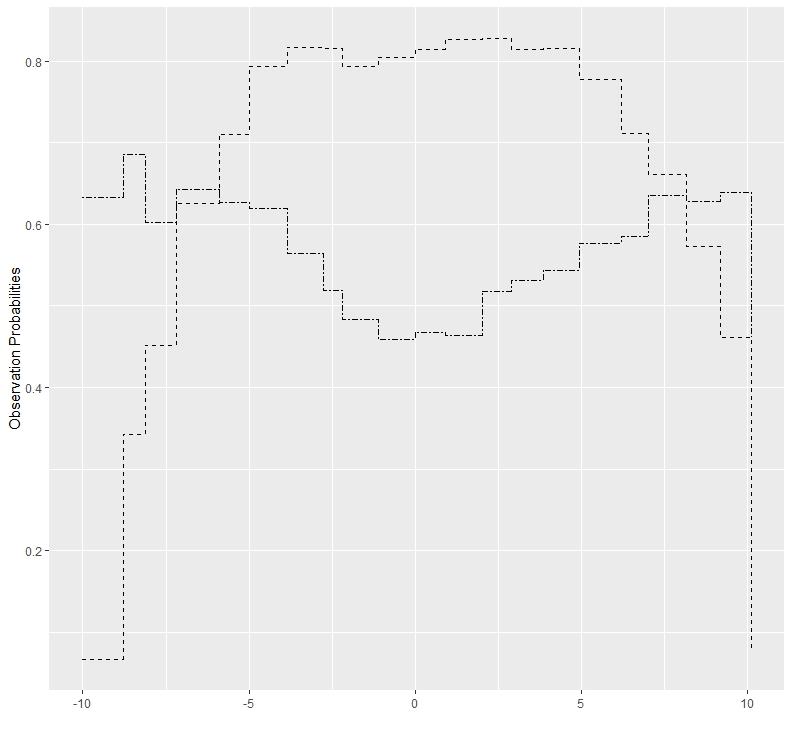}
		\caption{Estimation of $f$ in Example \ref{ex:real1}. The dashed curve corresponds to the modified USVT estimator and the double-dashed curve corresponds to the modified Cand\`es-Recht estimator.}
		\label{fig:figurejes}
	\end{center}
\end{figure}

\begin{figure}[tbp]
	\begin{center}
		\includegraphics[scale=0.4]{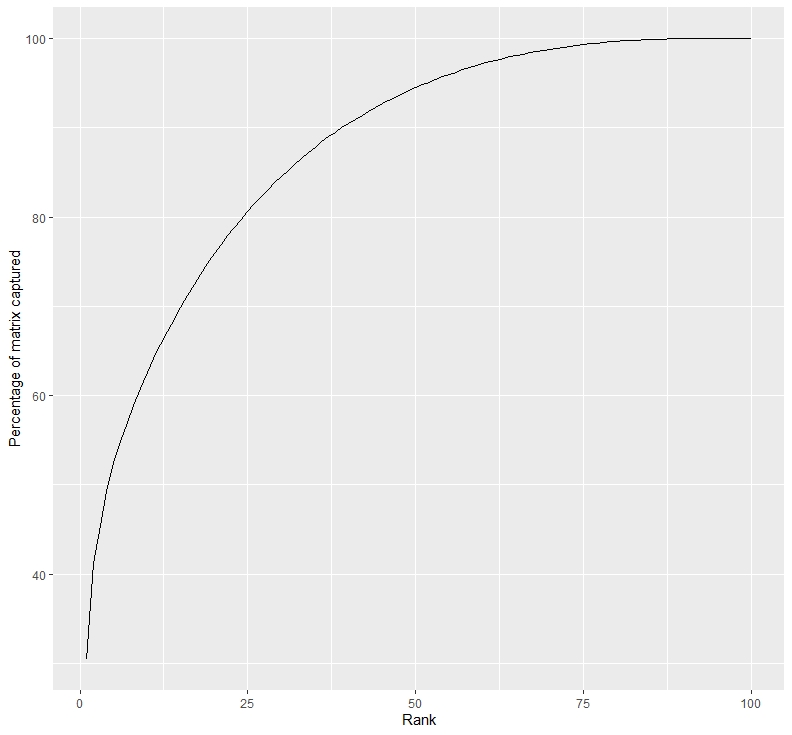}
		\caption{Let $\hat{M}$ be the modified Cand\`es--Recht estimate of $M$ in Example \ref{ex:real1}. This graph shows that percentage of $\hat{M}$ that is captured by its rank-$k$ approximation, $k = 1,2,\ldots,100$. }
		\label{fig:figurejes2}
	\end{center}
\end{figure}

\end{ex}

\begin{ex}\label{ex:real2}
We continue with the Jester data example. Assuming that the $\hat{f}$ given by the modified USVT estimator reflects the true state of affairs, we ran the modified USVT method on the whole dataset. The estimated $f$, with $b = 70$, is shown in Figure \ref{fig:figurefull}. The inverted U-shape is mysterious. It is not clear to us what may have led to this, if it is indeed close to the true $f$, because we do not know what caused entries to be missing in this dataset. 

\begin{figure}[tbp]
	\begin{center}
		\includegraphics[scale=0.4]{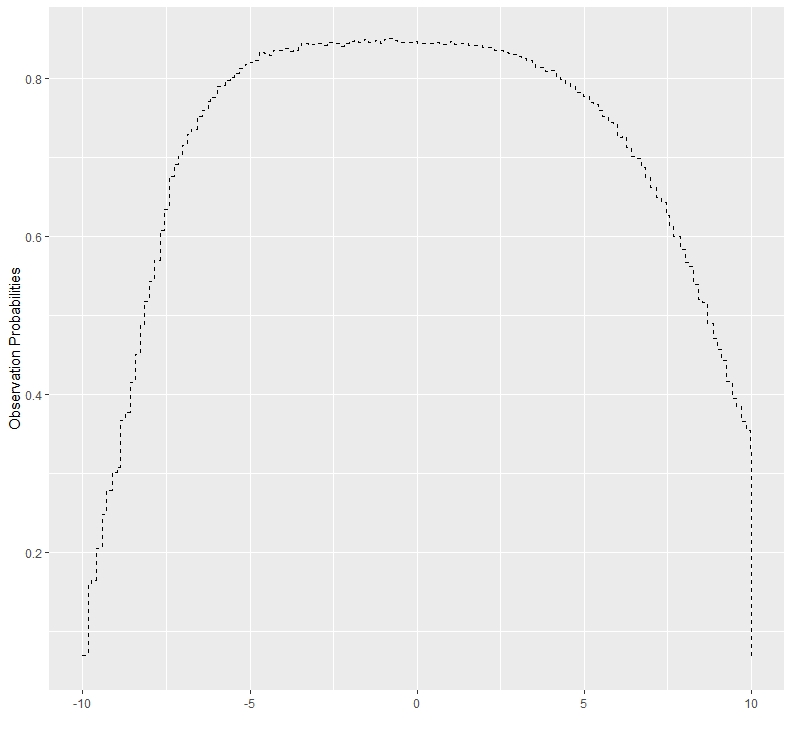}
		\caption{Estimation of $f$ in Example \ref{ex:real2} using the modified USVT estimator.}
		\label{fig:figurefull}
	\end{center}
\end{figure}
\end{ex}

\begin{ex}\label{ex:real3}
	For our final example, we consider the Film Trust dataset of movie ratings~\cite{filmrate}. This dataset consists of ratings given by $1508$ users to $2071$ movies, with many missing entries. The user ratings range in the set $\{0.5,1,1.5,2,2.5,3,3.5,4\}$. This dataset is much sparser than the Jester data; only $35497$ ratings are available, which is about $1.13$ percent of the total number of possible ratings. Due to the large size of the dataset, we implemented only the modified USVT algorithm. We assume that each user has a `true' rating for each movie, and the observed rating, if any, is a noisy version of the true rating. The observation probability is then a function $f$ of the true rating. The estimate of $f$, with $b=30$, is plotted in Figure \ref{fig:filmgraph}. As expected, a high rating increases the chance of the rating being available; however, there is a dip towards the end of the curve which we do not know how to explain. One possible explanation is that very highly rated movies are often classics that not many people watch and rate because they have already watched those movies before. 
	\begin{figure}[tbp]
		\begin{center}
			\includegraphics[scale=0.4]{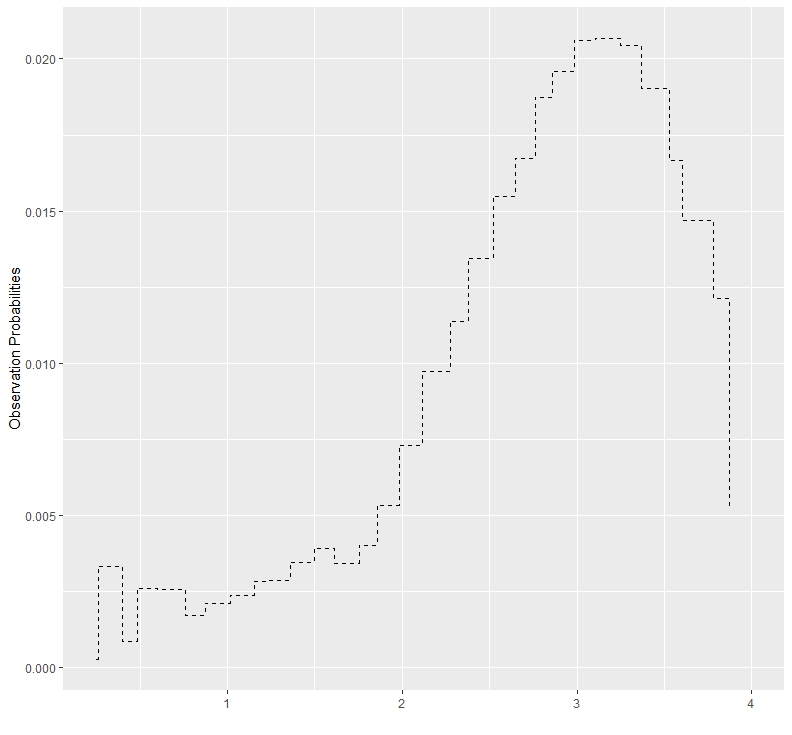}
			\caption{Estimation of $f$ in Example \ref{ex:real3} using the modified USVT estimator.}
			\label{fig:filmgraph}
		\end{center}
	\end{figure}
\end{ex}

\section{Proof of Theorem \ref{candesthm}}
For an $m \times n$ matrix $A$, define 
\[
\|A\|_2 := \biggl(\frac{1}{mn} \sum_{i,j} a_{ij}^2\biggr)^{1/2}.
\]
Note that $\|A\|_2^2$ is the sum of squares of the singular values of $A$, divided by $mn$. Given the matrix $A$, we will also denote by $A$ the function $A:[0,1]^2\to[-1,1]$ which equals $a_{ij}$ in the rectangle $(\frac{i-1}{m}, \frac{i}{m})\times (\frac{j-1}{n}, \frac{j}{n})$ for each $1\le i\le m$ and $1\le j\le n$. On the boundaries of the rectangles, we define the function $A$ is to be zero. Note that with this convention, $\|A\|_2$ equals the $L^2$ norm of the function $A$, which will also be denoted by $\|A\|_2$. 

For each $k$, let $S_k$ denote the group of all permutations of $\{1,\ldots,k\}$. Given an $m\times n$ matrix $A$ and a measurable map $W:[0,1]^2 \to [-1,1]$, we define
	\begin{equation}\label{eq:ml1}
	d_2(A,B):=\min_{\pi \in S_m,\, \tau \in S_n} \|A^{\pi,\tau}-W\|_2,
	\end{equation}
where $A^{\pi,\tau}$ is the matrix whose $(i,j)$-th entry is $a_{\pi(i)\tau(j)}$. 
The first key step in the proof of Theorem \ref{candesthm} is the following lemma. 
	\begin{lmm} \label{lmm:l1lem}
		Suppose that for each $k$, we have a matrix $M_k$ of order $m_k\times n_k$ with entries in $[-1,1]$, where $m_k,n_k\to \infty$ as $k\to\infty$. Suppose that this sequence has uniformly bounded rank. Then there exists a subsequence $M_{k_l}$ and a measurable map $W:[0,1]^2 \to [-1,1]$ such that $d_2(M_{k_l}, W)\to 0$ as~$l\to\infty$.
	\end{lmm}
	
		We will now prove Lemma \ref{lmm:l1lem}. The proof closely follows the proof of \cite[Theorem 1]{c19}. Let $m$ and $n$ be two positive integers. Let $\cp$ be a partition of $\{1,\ldots,m\}$ and let $\cq$ be a partition of $\{1,\ldots, n\}$. The pair $(\cp,\cq)$ defines a block structure for $m\times n$ matrices in the natural way: Two pairs of indices $(i,j)$ and $(i',j')$ belong to the same block if and only if $i$ and $i'$ belong to the same member of $\cp$ and $j$ and $j'$ belong to the same member of $\cq$. 
	
	If $A$ is an $m\times n$ matrix, let $A^{\cp,\cq}$ be the `block averaged' version of $A$, obtained by replacing the entries in each block (in the block structure defined by $(\cp,\cq)$) by the average value in that block. It is easy to see that 
	\begin{equation}\label{boxcontract}
	\|A^{\cp,\cq}\|_2 \le \|A\|_2.
	\end{equation}
	We need the following lemma.

	\begin{lmm}\label{szemlmm}
		For any $m\times n$ matrix $A$ with entries in $[-1,1]$, and $\rank(A)\le r$, there is a sequence of partitions $\{\cp_j\}_{j\ge 1}$ of $\{1,\ldots,m\}$ and a sequence of partitions $\{\cq_j\}_{j\ge 1}$ of $\{1,\ldots, n\}$ such that for each $j$,
		\begin{enumerate}
			\item $\cp_{j+1}$ is a refinement of $\cp_j$ and $\cq_{j+1}$ is a refinement of $\cq_j$, 
			\item $|\cp_j|$ and $|\cq_j|$ are bounded by $(2^{j+2}j)^{j^2}$, and
			\item $\|A - A^{\cp_j,\cq_j}\|_2\le 2\sqrt{r}/j+6j^3 2^{-j}$.
		\end{enumerate}
	\end{lmm}
	\begin{proof}
		Let $A = \sum_{i=1}^r \sigma_i u_i v_i^T$ be the singular value decomposition of $A$, where $\sigma_1\ge \cdots \ge \sigma_r$, and some of the $\sigma_i$'s are zero if the rank is strictly less than $r$. Take any $j\ge 1$. Let $l$ be the largest number such that $\sigma_l> \sqrt{mn}/j$. If there is no such $l$, let $l=0$. Let 
		\[
		A_1 := \sum_{i=1}^l \sigma_i u_i v_i^T.
		\]
		We define $\cp_j$, $\cq_j$, and $\tilde{A}_1$ as in the proof of \cite[Lemma 4]{c19}, as follows. For $1 \le i \le l$ and $1 \le a \le m$, let $u_{ia}$ denote the $a^{\text{th}}$ component of $u_i$. Let $\tilde{u}^{(j)}_{ia}$ be the largest integer multiple of $2^{-j}m^{-1/2}$ that is $\le u_{ia}$. Let $\tilde{u}^{(j)}_{i}$ be the vector whose $a^{\text{th}}$ component is $\tilde{u}^{(j)}_{ia}$. Similarly, for $1 \le b\le n$, let $\tilde{v}^{(j)}_{ib}$ be the largest integer multiple of $2^{-j}n^{-1/2}$ that is $\le v_{ib}$. Finally, define $\tilde{A}=\sum_{j=1}^{l} \sigma_i \tilde{u}_i \tilde{v}^\top_i$. This matrix $\tilde{A}$ is used as a block-approximation of $A_1$. As shown in \cite{c19}, this sequence of partitions satisfy property (1) and (2) in the statement of the lemma. Now, using the properties of the $\|\cdot \|_2$ norms noted earlier, and the facts that $l\le r$ and $\sigma_{l+1}\le \sqrt{mn}/j$, we have
		\begin{align*}
		\|A-A_1\|_2 &=\biggl(\frac{1}{mn} \sum_{i=l+1}^r \sigma_i^2\biggr)^{1/2} \le\biggl( \frac{r\sigma_{l+1}^2}{mn}\biggr)^{1/2}\le \frac{\sqrt{r}}{j}.
		\end{align*}
		Again, as in the proof of \cite[Lemma 4]{c19}, we obtain
		\begin{align*}
		\|A_1-\tilde{A}_1\|_2
		&\le 3j^3 2^{-j}. 
		\end{align*}
		Combining, we get
		\[
		\|A-\tilde{A}_1\|_2 \le \sqrt{r}/j+3j^3 2^{-j}. 
		\]
		Now note that $\tilde{A}_1$ is constant within the blocks defined by the pair $(\cp_j,\cq_j)$. Thus, by \eqref{boxcontract},
		\begin{align*}
		\|A-A^{\cp_j, \cq_j}\|_2 &\le\|A-\tilde{A}_1\|_2 + \|\tilde{A}_1-A^{\cp_j, \cq_j}\|_2\\
		&\le \|A-\tilde{A}_1\|_2 + \|\tilde{A}_1^{\cp_j,\cq_j}-A^{\cp_j, \cq_j}\|_2 \le 2\|A-\tilde{A}_1\|_2.
		\end{align*}
		This completes the proof.
	\end{proof}

	We are now ready to prove Lemma \ref{lmm:l1lem}. 
	\begin{proof}[Proof of Lemma \ref{lmm:l1lem}]
	Let $r$ be a uniform upper bound on the rank of $M_k$. Lemma~\ref{szemlmm} tells us that for each $k$ and $j$, we can find a partition $\cp_{k,j}$ of $\{1,\ldots,m_k\}$ and a partition $\cq_{k,j}$ of $\{1,\ldots, n_k\}$ such that
		\begin{enumerate}
			\item $\cp_{k,j+1}$ is a refinement of $\cp_{k,j}$ and $\cq_{k,j+1}$ is a refinement of $\cq_{k,j}$, 
			\item $|\cp_{k,j}|$ and $|\cq_{k,j}|$ are bounded by $(2^{j+2}j)^{j^2}$, and 
			\item $\|M_k - M_k^{\cp_{k,j},\cq_{k,j}}\|_2\le 2\sqrt{r}/j+6j^3 2^{-j}$.
		\end{enumerate}
		To reduce notation, let us denote $M_k^{\cp_{k,j},\cq_{k,j}}$ by $M_{k,j}$. Following the proof of \cite[Theorem 1]{c19} and passing to a subsequence if necessary, we get that for every $j$, there exists a measurable function $W_j:[0,1]^2 \to [-1,1]$ such that $M^{\pi_k,\tau_k}_{k,j}\to W_j$ in  $L^2$ as $k\to\infty$, where $\pi_k$ and $\tau_k$ are permutations that depend only on $k$ (and not on $j$). Without loss of generality, let us assume $\pi_k$ and $\tau_k$ are identity permutations for each $k$.
		
		By construction, the block structure for $W_{j+1}$ is a refinement of the block structure for $W_j$. Also by construction, the value of $W_j$ in one of its blocks is the average value of $W_{j+1}$ within that block. From this, by a standard martingale argument (for example, as in the proof of \cite[Theorem 9.23]{lovaszbook}) it follows that $W_j$ converges pointwise almost everywhere to a function $W$ as $j\to\infty$. In particular, $W_j\to  W$ in $L^2$. We claim that $M_k\to W$ in $L^2$ as $k\to\infty$. To show this, take any $\ve >0$. Find $j$ so large that $\|W-W_j\|_2 \le \ve$ and  $2\sqrt{r}/j+6j^3 2^{-j}\le \ve$. Then for any $k$,
		\begin{align*}
		\|W-M_k\|_2 &\le \|W-W_j\|_2 + \|W_j - M_{k,j}\|_2 + \|M_{k,j}-M_k\|_2\\
		&\le \ve + \|W_j - M_{k,j}\|_2 +  2\sqrt{r}/j+6j^3 2^{-j}\\
		&\le 2\ve + \|W_j - M_{k,j}\|_2.
		\end{align*}
		Since $M_{k,j}\to W_j$ in $L^2$ as $k\to\infty$ and $\ve$ is arbitrary, this completes the proof. 
	\end{proof}

	Henceforth, let us work in the setting of Theorem \ref{candesthm}. For each $k$, let $P_k$ be the random binary matrix whose $(i,j)$-the entry is $1$ if the $(i,j)$-th entry of $M_k$ is revealed, and $0$ otherwise. Then note that as  functions on $[0,1]^2$, $\mathbb{E}(P_k)=f\circ M_k$, where $\ee(P_k)$ denotes the matrix of expected values of the entries of $P_k$.
	
	Recall the {\it cut norm} on the set of $m\times n$ matrices, as defined in \cite{c19}:
	\[
	\|A\|_\Box := \frac{1}{mn}\max\{|x^T A y| :x\in \rr^m,\, y\in \rr^n,\, \|x\|_\infty \le 1, \, \|y\|_\infty\le 1\},
	\]
	where $\|x\|_\infty$ denotes the $\ell^\infty$ norm of a vector $x$. 
	If $A$ is an $m\times n$ matrix and $W:[0,1]^2\to \rr$ is a measurable function, we define $d_\Box(A, W)$ to be $\|A-B\|_\Box$, where $B$ is the $m\times n$ matrix whose $(i,j)$-th entry is the average value of $W$ in the rectangle $(\frac{i-1}{m}, \frac{i}{m})\times (\frac{j-1}{n}, \frac{j}{n})$. 
	
	The following lemma shows that $P_k$ and $\ee(P_k)$ are close in cut norm.
	\begin{lmm}\label{plmm}
	As $k\to\infty$, $\|P_k-\mathbb{E}(P_k)\|_\Box \rightarrow 0$ in probability. 
	\end{lmm}
	\begin{proof}
	It is easy to see from the definition of cut norm that for an $m \times n$ matrix $A$,
	$$\|A\|_\Box \leq \frac{\|A\|_{op}}{\sqrt{mn}},$$
	where $\|A\|_{op}$ is the $\ell^2$ operator norm of $A$. Now take any $t>0$. Using \cite[Theorem 3.4]{usvt}, $\mathbb{P}(\|P_k-\mathbb{E}(P_k)\| \geq 3\sqrt{n_k}) \leq C_1 e^{-C_2 n_k}$ for some positive universal constants $C_1$ and $C_2$. Hence, for $k$ large enough,
	\begin{align*}
	\mathbb{P}(\|P_k-\mathbb{E}(P_k)\|_\Box \geq t) &\leq \mathbb{P}(\|P_k-\mathbb{E}(P_k)\|_{op} \geq t\sqrt{m_k n_k})\\
	& \leq \mathbb{P}(\|P_k-\mathbb{E}(P_k)\|_{op} \geq 3\sqrt{n_k}) \leq C_1 e^{-C_2 n_k}.
	\end{align*}
	This shows that $\|P_k-\mathbb{E}(P_k)\|_\Box \to 0$ in probability as $k\to\infty$. 
	\end{proof}
	
Next, we relate the limiting empirical distribution of the entries of $M_k$ with the $L^2$ limit of $M_k$ as a function on $[0,1]^2$. In the following,  $\lambda$ denotes Lebesgue measure on $[0,1]^2$.
\begin{lmm}\label{mkwlmm1}
Suppose that $M_k \to W$ in $L^2$ as a sequence of functions on $[0,1]^2$. Then $\mu_k$ converges weakly to $\mu = \lambda \circ W^{-1}$. 
\end{lmm}
\begin{proof}
	Take any bounded continuous function $g:[-1,1]\to \rr$. It is not difficult to see that
	\[
	\int g d\mu_k = \iint g(M_k(x,y))dxdy.
	\]
	Since $M_k\to W$ in $L^2$ and $g$ is bounded and continuous, we get
\begin{align*}
\lim_{k\to\infty}  \iint g(M_k(x,y))dxdy &= \iint g(W(x,y)) dxdy.
\end{align*}
But the right side is the integral of $g$ with respect to the measure $\lambda\circ W^{-1}$. This completes the proof.
\end{proof}
The purpose of the next lemma is to investigate the convergence of $f\circ M_k$ under the hypotheses of Theorem \ref{candesthm}.
\begin{lmm}\label{mkwlmm2}
Suppose that $M_k \to W$ in $L^2$ as a sequence of functions on $[0,1]^2$. Let $\mu := \lambda \circ W^{-1}$. Suppose that $g:[-1,1]\to [0,1]$ is a measurable function which is continuous almost everywhere with respect to $\mu$. Then $g\circ M_k \to g\circ W$ in $L^2$. 
\end{lmm}
\begin{proof}
Since $M_k\to W$ in $L^2$, any subsequence has a further subsequence along which $M_k(x,y) \to W(x,y)$ for $\lambda$-a.e.~$(x,y)$. By assumption, $g$ is continuous at $W(x,y)$ for $\lambda$-a.e.~$(x,y)$. Combining these two observations, we get that for any subsequence, there is a further subsequence along with $g\circ M_k(x,y) \to g\circ W(x,y)$ for $\lambda$-a.e.~$(x,y)$. Since $g$, $M_k$ and $W$ are all taking values in $[0,1]$, this implies that $g\circ M_k \to g\circ W$ in $L^2$ along this subsequence. This completes the proof.
\end{proof}
As a consequence of the above lemmas, we obtain the following result.
\begin{lmm}\label{pconvlmm}
Suppose that $M_k \to W$ in $L^2$ as a sequence of functions on $[0,1]^2$. Then, under the hypotheses of Theorem \ref{candesthm}, there is a measurable function $V:[0,1]^2 \to [0,1]$ that is nonzero almost everywhere and $d_\Box(P_k,V) \to 0$ in probability as $k\to \infty$.
\end{lmm}
\begin{proof}
By Lemma \ref{plmm}, it suffices to show that $d_\Box(\ee(P_k),V)\to 0$ for some $V$ as in the statement of the lemma. By Lemma \ref{mkwlmm1}, $\mu_k$ converges weakly to $\mu=\lambda \circ W^{-1}$. By the hypotheses of Theorem \ref{candesthm}, there is a measurable extension of $f$ to $[-1,1]$, also denoted by $f$, which is nonzero and continuous $\mu$-a.e. As noted earlier, $\ee(P_k) = f\circ M_k$. Therefore, by Lemma \ref{mkwlmm2}, $\ee(P_k)\to f\circ W$ in $L^2$. It is not hard to see that this implies that $d_\Box(\ee(P_k), f\circ W)\to 0$. But $f\circ W$ is nonzero $\lambda$-a.e. Thus, we can take $V= f\circ W$.
\end{proof}

We are now ready to prove Theorem \ref{candesthm}.
\begin{proof}
Suppose that $\hat{M}_k$ is not a consistent sequence of estimators. Then, passing to a subsequence if necessary, we may assume that 
\begin{align}\label{mkineq}
\inf_{k\ge 1}  \ee\|\hat{M}_k - M_k\|_2^2 >0.
\end{align}
Note that this condition continues to hold true if we pass to further subsequences and permute rows and columns in each $M_k$, which we will do shortly. Passing to a further subsequence, and permuting rows and columns in each $M_k$ if necessary, we use Lemma \ref{lmm:l1lem} to get an $L^2$ limit $W$ of $M_k$ as $k\to \infty$. Then, by Lemma \ref{pconvlmm}, there is a measurable function $V:[0,1]^2 \to [0,1]$ that is nonzero almost everywhere and $d_\Box(P_k,V) \to 0$ in probability as $k\to \infty$. Again passing to a subsequence, we get that $d_\Box(P_k,V) \to 0$ almost surely. But this implies, by \cite[Theorem 2 and Theorem 3]{c19}, that $\|\hat{M}_k - M_k\|_2\to 0$ almost surely. Since the entries of $M_k$ and $\hat{M}_k$ are in $[-1,1]$ for all $k$, this contradicts \eqref{mkineq}. 
\end{proof}


\section{Proof of Theorem \ref{usvtthm}}
Without loss of generality, suppose that $m_k\le n_k$ for each $k$. (Otherwise, we can just transpose the matrices.) Let $r$ be a uniform upper bound on the rank of $M_k$. Let $R_k$ be the matrix obtained by applying $f$ entrywise to $M_k$. Let $Q_k$ be the entrywise (i.e., Hadamard) product of $M_k$ and $R_k$. Let $Y_k$ be the matrix obtained by replacing the unrevealed entries of $X_k$ by zero. Let $P_k$ be the matrix whose $(i,j)$-th entry is $1$ if the $(i,j)$-th entry of $X_k$ is revealed, and $0$ otherwise. Note that $\ee(Y_k)=Q_k$ and $\ee(P_k)=R_k$. Note also that the entries of $Y_k$ and $P_k$ are all in $[-1,1]$. 

First, let us assume that $\mu_k$ converges weakly to a limit $\mu$ as $k\to \infty$. Then by the hypotheses of Theorem \ref{usvtthm}, $f$ has an extension to a Lipschitz function on $[-1,1]$, also called $f$, which has no zeros in the support of $\mu$.  Let us fix such an extension, and let $L$ denote its Lipschitz constant. 
\begin{lmm}\label{rlmm}
As $k\to\infty$, $\|R_k\|_*= o(m_k\sqrt{n_k})$. 
\end{lmm}
\begin{proof}
	Fix $\varepsilon >0$. It is an easy consequence of the Cauchy--Schwarz inequality that for any $k$,
	\[
	\|M_k\|_*\le \|M_k\|_2  \sqrt{\rank(M_k) m_kn_k}\le \sqrt{rm_kn_k}.
	\]
	By \cite[Lemma 2]{c19}, this implies that there is a block matrix $B_k$ with at most $b$ blocks, where $b$ depends only on $\ve$ and $r$, and entries in $[-1,1]$, such that $\|M_k - B_k\|_2 \le \ve$. Let $D_k$ be obtained by applying $f$ to $B_k$ entrywise. Then by the Lipschitz property of $f$, we get
	\[
	\|R_k - D_k\|_2\le \ve L. 
	\]
	Note that just like $B_k$, $D_k$ has at most $b$ blocks. In particular, $\rank(D_k)\le b$. Therefore again by the Cauchy--Schwarz inequality,
	\begin{align*}
	\|R_k\|_* &\le \|R_k-D_k\|_* + \|D_k\|_*\\
	&\le \|R_k - D_k\|_2 \sqrt{\rank(R_k-D_k)m_kn_k} + \|D_k\|_2  \sqrt{\rank(D_k) m_kn_k}\\
	&\le \ve L m_k\sqrt{n_k}  + \sqrt{bm_kn_k}.
	\end{align*}
	Thus, 
	\[
	\limsup_{k\to \infty} \frac{\|R_k\|_*}{m_k \sqrt{n_k}} \le \ve L.
	\]
	Since this holds for arbitrary $\ve >0$, this completes the proof.
	\end{proof}

\begin{lmm}\label{qlmm}
As $k\to\infty$, $\|Q_k\|_*= o(m_k\sqrt{n_k})$. 
\end{lmm}
\begin{proof}
Let $B_k$, $b$, and $D_k$ be as in Lemma \ref{qlmm}. Let $E_k$ be the Hadamard product of $B_k$ and $D_k$, and $F_k$ be the Hadamard product of $B_k$ and $R_k$. Then $E_k$ also has $b$ blocks. Moreover, since the entries of all these matrices are in $[-1,1]$, it is not hard to see that 
\begin{align*}
\|Q_k - E_k\|_2 &\le \|Q_k - F_k\|_2 + \|F_k - E_k\|_2\\
&\le  \|M_k - B_k\|_2 + \|R_k - D_k\|_2\\
&\le (L+1)\ve.
\end{align*}
The rest of the proof is the same as the proof of Lemma \ref{rlmm}, with $R_k$ replaced by $Q_k$ and $D_k$ replaced by $E_k$.
\end{proof}
As a consequence of the above lemmas, we obtain the following result.
\begin{lmm}\label{usvtlmm}
Let $\hat{Q}_k$ and $\hat{R}_k$ be the estimates of $Q_k$ and $R_k$ obtained by applying the USVT algorithm to $Y_k$ and $P_k$. Then $\ee\|\hat{Q}_k - Q_k\|_2^2 \to 0$ and $\ee\|\hat{R}_k - R_k\|_2^2 \to 0$ as $k\to \infty$. 
\end{lmm}
\begin{proof}
This is a direct consequence of Lemmas \ref{rlmm} and \ref{qlmm} and the consistency of the USVT estimator~\cite[Theorem 1.1]{usvt}.
\end{proof}

Let us now prove Theorem \ref{usvtthm} under the simplifying assumption under which we are currently working. Let $\hat{M}_k$ denote the modified USVT estimator. Let $m_{kij}$ denote the $(i,j)$-th element of $M_k$, $\hat{m}_{kij}$ denote the $(i,j)$-th element of $\hat{M}_k$, etc.

Since $f$ is nonzero and continuous on the support of $\mu$, and the support is a compact set, there exists $\delta>0$ such that $f> \delta$ everywhere on the support of $\mu$. In particular, $\mu(\{x: f(x)\le \delta\}) = 0$. Since $\mu_k\to \mu$ weakly, and $\{x:f(x)\le \delta\}$ is a closed set due to the continuity of $f$, we get
\begin{align*}
\limsup_{k\to \infty} \mu_k(\{x: f(x)\le \delta\}) \le \mu(\{x: f(x)\le \delta\}) = 0.
\end{align*} 
In other words, if we let $I_k:= \{(i,j): r_{kij}\le \delta\}$, then $|I_k| = o(m_kn_k)$ as $k\to\infty$. 

Let $J_k := \{(i,j): \hat{r}_{kij} \le \delta/2\}$. Then 
\begin{align*}
|J_k| &\le |I_k| + |\{(i,j): |\hat{r}_{kij} - r_{kij}| > \delta/2\}|\\
&\le |I_k| + \frac{4}{\delta^2}\sum_{i,j} (\hat{r}_{kij} - r_{kij})^2\\
&= |I_k| + \frac{4m_kn_k}{\delta^2} \|\hat{R} - R\|_2^2.
\end{align*}
By Lemma \ref{usvtlmm} and the fact that $|I_k| = o(m_kn_k)$, this shows that $|J_k|=o_P(m_kn_k)$ as $k\to\infty$ (meaning that $|J_k|/(m_kn_k) \to 0$ in probability as $k\to\infty$). 

Now take $(i,j)\notin I_k \cup J_k$. Then 
\begin{align*}
\biggl|\frac{\hat{q}_{kij}}{\hat{r}_{kij}} - m_{kij} \biggr| &= \biggl|\frac{\hat{q}_{kij}}{\hat{r}_{kij}} - \frac{q_{kij}}{r_{kij}} \biggr|\\
&\le \frac{|\hat{q}_{kij} - q_{kij}|}{\hat{r}_{kij}} + \frac{|q_{kij}||\hat{r}_{kij} - r_{kij}|}{\hat{r}_{kij} r_{kij}}\\
&\le \frac{2}{\delta} |\hat{q}_{kij} - q_{kij}| + \frac{2}{\delta^2}|\hat{r}_{kij} - r_{kij}|.
\end{align*}
Since $\hat{m}_{kij}$ is obtained by truncating $\hat{q}_{kij}/\hat{r}_{kij}$, the above upper bound also holds for $|\hat{m}_{kij} - m_{kij}|$ when $(i,j)\notin I_k \cup J_k$. But $|\hat{m}_{kij} - m_{kij}|\le 2$ for any $(i,j)$. Thus,
\begin{align*}
\sum_{i,j} (\hat{m}_{kij} - m_{kij})^2 &\le 4|I_k\cup J_k| + \sum_{(i,j)} \biggl( \frac{2}{\delta} |\hat{q}_{kij} - q_{kij}| + \frac{2}{\delta^2}|\hat{r}_{kij} - r_{kij}|\biggr)^2\\
&\le  4|I_k\cup J_k| + \frac{8}{\delta^2} \sum_{i,j} (\hat{q}_{kij} - q_{kij})^2 \\
&\qquad + \frac{8}{\delta^4} \sum_{i,j} ( \hat{r}_{kij} - r_{kij})^2.
\end{align*}
By Lemma \ref{usvtlmm} and our previous deduction that $|I_k\cup J_k| = o_P(m_kn_k)$, the above inequality shows that $\|\hat{M}_k - M_k\|_2 \to 0$ in probability as $k\to \infty$. Since this is a uniformly bounded sequence of random variables, this proves the consistency of $\hat{M}_k$. This proves Theorem \ref{usvtthm} under the simplifying assumption that $\mu_k$ converges weakly to some $\mu$ as $k\to\infty$. We are now ready to prove Theorem~\ref{usvtthm} in full generality.
\begin{proof}[Proof of Theorem \ref{usvtthm}]
Let $\hat{M}_k$ be the modified USVT estimator of $M_k$. Suppose that $\{\hat{M}_k\}_{k\ge 1}$ is not a consistent sequence of estimators. Passing to a subsequence if necessary, we may assume that
\begin{align}\label{mkineq2}
\inf_{k\ge 1}  \ee\|\hat{M}_k - M_k\|_2^2 >0.
\end{align}
Note that this will continue to hold true if we pass to further subsequences. Passing to a further subsequence, we may assume that $\mu_k$ converges weakly to some $\mu$. But then we already know that \eqref{mkineq2} is violated. This completes the proof of the theorem.
\end{proof}

\section{Proof of Theorem \ref{fthm}}
	In this proof, $C$ will denote any universal constant, whose value may change from line to line. Let $[x,y)$ be a subinterval of $[-2,2]$. Let $p_{ij}=1$ if the $(i,j)$-th entry of $M$ is revealed and $0$ otherwise. Let \[
S_{x,y} := \{(i,j) : m_{ij}\in [x,y)\}, \ \ T_{x,y} := \{(i,j): \hat{m}_{ij}\in [x,y)\},
\]
and let 
	\[
	\hat{f}_{x,y} := \frac{1}{|T_{x,y}|} \sum_{(i,j)\in T_{x,y}} p_{ij}, \ \ g_{x,y} := \frac{1}{|S_{x,y}|} \sum_{(i,j)\in S_{x,y}} p_{ij}
	\]
	where the right sides are declared to be zero if the corresponding sums are empty. Note that $\hat{f}_{x,y}$ and $g_{x,y}$ are always in $[0,1]$. 
	Take some $\delta < (y-x)/2$, to be chosen later. Let 
	\begin{align}\label{mudef}
	\mu_{x,y} := \frac{1}{mn}|\{(i,j): m_{ij}\in [a-\delta, a+\delta]\cup[b-\delta, b+\delta]\}|.
	\end{align}
Take any $(i,j)\in T_{x,y}\setminus S_{x,y}$. There are two cases. First suppose that $m_{ij}\notin [x-\delta, y+\delta]$. Since $(i,j)\in T_{x,y}$, we have $\hat{m}_{ij}\in [x,y)$, and hence in this case, $|\hat{m}_{ij} - m_{ij}|>\delta$. By Markov's inequality, the number of such $(i,j)$ is bounded above by 
	\begin{align}\label{msedelta}
	\frac{1}{\delta^2}\sum_{i=1}^m \sum_{j=1}^n (m_{ij} - \hat{m}_{ij})^2.
	\end{align}
	The second case is that $m_{ij}\in [x-\delta, x) \cup [y, y+\delta]$. By the definition of $\mu_{x,y}$, the number of such $(i,j)$ is at most $mn\mu_{x,y}$. Combining, we get that
	\[
	|T_{x,y}\setminus S_{x,y}|\le \frac{1}{\delta^2}\sum_{i=1}^m \sum_{j=1}^n (m_{ij} - \hat{m}_{ij})^2 + mn\mu_{x,y}.
	\]
	Now take any $(i,j)\in S_{x,y}\setminus T_{x,y}$. Then, again, there are two cases. First, suppose that $m_{ij}\in [x+\delta, y-\delta]$. Since $\hat{m}_{ij}\notin [x,y)$, in this case we have that $|\hat{m}_{ij}-m_{ij}|> \delta$. Thus, by Markov's inequality, the number of such $(i,j)$ is bounded above by \eqref{msedelta}. The other case is $m_{ij}\in [x, x+\delta)\cup (y-\delta, y)$. As before, the number of such $(i,j)$ is bounded above by $mn\mu_{x,y}$. Combining these two observations, we get that 
	\begin{align}\label{tdeltas}
	\ee|T_{x,y}\Delta S_{x,y}| &\le \frac{2mn\theta}{\delta^2} + 2mn\mu_{x,y}. 
	\end{align}
	We will now work under the assumption that $S_{x,y}\ne \emptyset$. The final estimate will be valid even if $S_{x,y}=\emptyset$. 
	First, note that
	\begin{align*}
	\var(g_{x,y})  &= \frac{1}{|S_{x,y}|^2} \sum_{(i,j)\in S_{x,y}} \var(Y_{ij}) \le \frac{1}{4|S_{x,y}|}. 
	\end{align*}
	Let $f_{x,y} :=\ee(g_{x,y})$. Then the above bound can be written as 
	\begin{align}\label{fbd1}
	|S_{x,y}| \ee[(g_{x,y}-f_{x,y})^2] \le \frac{1}{4}.
	\end{align}
	Clearly, the above bound holds even if $S_{x,y}=\emptyset$. 
	Next, note that 
	\begin{align*}
	|g_{x,y} - \hat{f}_{x,y} | &\le \frac{1}{|S_{x,y}|} \biggl|\sum_{(i,j)\in S_{x,y}} Y_{ij} - \sum_{(i,j)\in T_{x,y}} Y_{ij}\biggr| \\
	&\qquad + \biggl|\frac{1}{|S_{x,y}|}-\frac{1}{|T_{x,y}|}\biggr| \sum_{(i,j)\in T_{x,y}} Y_{ij}\\
	&\le \frac{1}{|S_{x,y}|} \sum_{(i,j)\in T_{x,y}\Delta S_{x,y}} Y_{ij} + \frac{||T_{x,y}|-|S_{x,y}||}{|S_{x,y}|}\\
	&\le \frac{2|T_{x,y}\Delta S_{x,y}|}{|S_{x,y}|}. 
	\end{align*}
	This shows, by \eqref{tdeltas} and the fact that $\hat{f}_{x,y}$ and $g_{x,y}$ are both in $[0,1]$, that  
	\begin{align}
	|S_{x,y}|\ee[(\hat{f}_{x,y} - g_{x,y})^2] &\le |S_{x,y}|\ee|\hat{f}_{x,y} - g_{x,y}| \notag\\
	&\le 2\ee|T_{x,y}\Delta S_{x,y}| \le  \frac{4mn\theta}{\delta^2} + 4mn\mu_{x,y}.\label{fbd2}
	\end{align}
	Again, this bound holds even if $S_{x,y}=\emptyset$. Combining \eqref{fbd1} and \eqref{fbd2}, we get
	\begin{align}\label{fbd3}
	|S_{x,y}|\ee[(\hat{f}_{x,y} - f_{x,y})^2 ] &\le \frac{8mn\theta}{\delta^2} + 8mn\mu_{x,y} + \frac{1}{4}.
	\end{align}
	Using the notation \eqref{mudef}, we see that for any $1\le l\le b+2$, 
	\begin{align*}
	\ee(\mu_{a_l, a_{l+1}}) &= \frac{1}{mn}\sum_{i=1}^m \sum_{j=1}^n[\pp(|m_{ij}-a_l|\le\delta) + \pp(|m_{ij}-a_{l+1}|\le \delta)]\\
	&\le 8b\delta. 
	\end{align*}
	Applying \eqref{fbd3} to the interval $[a_l,a_{l+1})$, taking expectation over the randomness of the $a_l$'s and applying the above inequality, and then summing over $l$, we get
	\begin{align*}
	\frac{1}{mn}\sum_{l=1}^{b+2}|S_{a_l, a_{l+1}}| \ee[(\hat{f}_{a_l,a_{l+1}} - f_{a_l,a_{l+1}})^2 ] &\le \frac{C\theta b}{\delta^2} + C b^2 \delta + \frac{Cb}{mn}. 
	\end{align*}
	Choosing $\delta = (\theta/b)^{1/3}$ gives
	\begin{align*}
	\frac{1}{mn}\sum_{l=1}^{b+2}|S_{a_l, a_{l+1}}| \ee[(\hat{f}_{a_l,a_{l+1}} - f_{a_l,a_{l+1}})^2 ] &\le C\theta^{1/3}b^{5/3} + \frac{Cb}{mn}. 
	\end{align*}
For $x\in [a_l,a_{l+1})$, let
\begin{align*}
\tilde{f}(x) := \frac{1}{|S_{a_l,a_{l+1}}|} \sum_{(i,j)\in S_{a_l, a_l+1}} f(m_{ij})
\end{align*}
Then note that for any $(i,j)\in S_{a_l,a_{l+1}}$, 
\begin{align*}
|\tilde{f}(m_{ij})-f(m_{ij})| &\le \frac{CL}{b}. 
\end{align*}
Since 
\begin{align*}
\frac{1}{mn}\sum_{l=1}^{b+2}\ee[|S_{a_l, a_{l+1}}|(\hat{f}_{a_l,a_{l+1}} - f_{a_l,a_{l+1}})^2 ]  &= \frac{1}{mn}\sum_{i,j} \ee[(\hat{f}(m_{ij}) - \tilde{f}(m_{ij}))^2],
\end{align*}
this completes the proof of the theorem.

\section*{Acknowledgement}
S.~B. thanks Debangan Dey, Samriddha Lahiry, Samyak Rajanala  and Subhabrata Sen for helpful discussions. S.~C.'s research was partially supported by NSF grant DMS-1855484. Lastly, we thank the two anonymous referees for their useful suggestions. 

\bibliographystyle{plainnat}
\bibliography{abc}

\end{document}